\documentclass[12pt]{amsart}
\usepackage[english]{babel}

\usepackage{amsmath,amsthm,amssymb}
\usepackage{amsfonts}
\usepackage[width=16.00cm, height=22.00cm]{geometry}
	 
\newcommand{\N}{\mathbb{N}}

\newcommand{\C}{\mathbb{C}}

\newcommand{\PmX}{\mathcal{P}(^m X)}

\newtheorem{teo}{Theorem}[section]
\newtheorem{lema}[teo]{Lemma}

\newtheorem{coro}[teo]{Corollary}
\newtheorem{prop}[teo]{Proposition}
\theoremstyle{definition}
\newtheorem{examp}[teo]{Example}
\newtheorem{nota}[teo]{Remark}
\title{Mean ergodic composition operators in spaces of homogeneous polynomials}

\author{David Jornet \and Daniel Santacreu \and Pablo Sevilla-Peris}



\begin{document}
\begin{abstract}
We study some dynamical properties of composition operators defined on the space $\PmX$ of $m$-homogeneous polynomials on a Banach space $X$ when $\PmX$ is endowed with two different topologies: the one of uniform convergence on compact sets and the one defined by the usual norm. The situation is quite different for both topologies: while in the case of uniform convergence on compact sets every power bounded composition operator is uniformly mean ergodic, for the topology of the norm there is no relation between the latter properties. Several examples are given.
\end{abstract}
\maketitle

\section{Introduction}

If $X$ is a complex Banach space and $\varphi : X \to X$ is holomorphic (all needed definitions are given below), then $C_{\varphi} : H(X) \to H(X)$, the \textit{composition operator of symbol $\varphi$}, is defined as $C_{\varphi}(f) = f \circ \varphi$. In this note 
we deal with the restriction of such an operator to the space $\mathcal{P}(^{m}X)$ of $m$-homogeneous polynomials, and we ask different questions. In first place, for which $\varphi$'s does this restriction take values again in $\mathcal{P}(^{m}X)$ (in other words, 
we have $C_{\varphi} : \mathcal{P}(^{m}X) \to \mathcal{P}(^{m}X)$ is well defined). Once we have settled this question (see Proposition~\ref{Linealsiosi}), we   study certain properties related with the linear dynamics of the composition operator (that is, with the behaviour of the iterated composition of the operator with itself): \textit{power boundedness} and \textit{mean ergodicity}.

We begin
by fixing some notation and basic notions. Given complex Banach spaces $X$ and $Y$, a mapping $p:X \to Y$ is an $m$-homogenous polynomial if there exists a continuous $m$-linear operator $L : X \times \cdots \times X \to Y$ so that $p(x) = L(x, \ldots , x)$ for every $x \in X$. The vector space of all $m$-homogenous polynomials is denoted by $\mathcal{P}(^{m} X, Y)$, and $\mathcal{P}(^{m} X)$ whenever $Y=\mathbb{C}$. Note that $\mathcal{P}(^{1}X)$ is nothing else than the topological dual of $X$, which we denote by $X'$. A function $f :X \to Y$ is holomorphic if 
there exists a (unique) sequence $(p_{m})_{m}$, where each $p_{m} : X \to Y$ is an $m$-homogeneous polynomial which satisfies
\begin{equation} \label{taylor}
f(x) = \sum_{m=0}^{\infty} p_{m} (x)
\end{equation}
uniformly on the compact sets of $X$. The space of all holomorphic functions $f:X \to Y$ is denoted by $H(X,Y)$. Again, we write $H(X)$ for $H(X, \mathbb{C})$. 

The space $\mathcal{P}(^{m} X)$ can be endowed with  different topologies. Here, we consider two of them. On the one hand, we consider on $\mathcal{P}(^{m} X)$ the compact-open topology, i.e. the topology of uniform convergence on the compact subsets of $X$. In this case, we denote the space by $\mathcal{P}(^{m} X)_{\tau_{0}}$. On the other hand, given $p \in  \mathcal{P}(^{m}X)$, we  define the norm
\begin{equation} \label{norma-p}
\Vert p \Vert = \sup_{\Vert x \Vert_{X} \leq 1} \vert p(x) \vert < \infty \,,
\end{equation}
which  turns $\mathcal{P}(^{m}X)$ into a Banach space,  that we denote by $\mathcal{P}(^{m}X)_{\Vert \cdot \Vert}$.

If $E$ is a locally convex Hausdorff space (lcHs), the space of all continuous linear  operators $T:E \to E$ is denoted by $L(E)$.  Given $T\in L(E)$, the iterates of $T$ are denoted by $T^n:=T\circ\dots\circ T$, for $n\in\N$. The operator $T$ is said to be \textit{power bounded} if the sequence $(T^n)_{n\in\N} \subseteq L(E)$ is equicontinuous.  
The Ces\`aro means of $T$ are given by
\[ 
T_{[n]}:= \frac{1}{n}\sum_{k=0}^{n-1} T^k,
\]
for $n \in \mathbb{N}$. Then $T$ is said to be \textit{mean ergodic} if $(T_{[n]} x)_{n}$ converges in $E$ for every $x\in E$. Moreover, an operator $T$ is  \emph{uniformly mean ergodic} if $(T_{[n]})_{n\in\N}$ converges to some operator $S\in L(E)$ in the topology of uniform convergence on the bounded subsets of $E$, and \textit{Ces\`aro bounded} if the sequence $(T_{[n]})_{n}$ is equicontinuous. A simple computation shows that every power bounded operator is Ces\`aro bounded. 

This note is motivated by several previous works existing in the literature. We mention \cite{BoDo2011A} where the authors characterise those composition operators $C_{\varphi}:H(U)\to H(U)$ which are power bounded when defined on the space of holomorphic functions $H(U)$ on a connected domain of holomorphy $U$ of $\C^{d}$. Moreover, it is proved in \cite{BoDo2011A} that $C_{\varphi}$ is power bounded if and only if it is (uniformly) mean ergodic if and only if the symbol $\varphi$ has stable orbits. If the domain is the unit disc, the authors in \cite{BGJJ2016m} characterise when $C_{\varphi}$ is mean ergodic or uniformly mean ergodic on the disc algebra or on the space of bounded holomorphic functions in terms of the asymptotic behaviour of the symbol. In \cite{BGJJ2016mw} it is investigated the power boundedness and (uniform) mean ergodicity of weighted composition operators on the space of holomorphic functions on the unit disc in terms of the symbol and the  multiplier. Finally, in \cite{K2019p}  the author studies power boundedness and  mean ergodicity for (weighted) composition operators on function spaces defined by local properties in a very general framework which extends previous works. In particular, permits to characterize (uniform) mean ergodicity for composition operators on a large class of function spaces which are Fr\'echet-Montel spaces when equipped with the compact-open topology. The space $\PmX$ is neither Fr\'echet with the compact-open topology, nor Montel in the Banach case. Hence, the results of \cite{K2019p} do not apply in our setting. 

The paper is organised as follows. In Section 2 we see, in Proposition~\ref{Linealsiosi}, that the composition operator $C_\varphi:\PmX\rightarrow\PmX$ is well defined only when its corresponding symbol $\varphi$ is linear. In Section 3 we study the dynamics of $C_{\varphi}$ on $\PmX_{\tau_{0}}$. We characterise when $C_{\varphi}$ is power bounded (Proposition~\ref{PBDestablepolinomis}) in the spirit of \cite[Proposition 1]{BoDo2011A}. Moreover, since the space $\PmX_{\tau_{0}}$ is semi-Montel, we can show in Corollary~\ref{PmX powerbdd ume} that every power bounded composition operator $C_{\varphi}$ is uniformly mean ergodic. We finish this section by giving an example of a composition operator which is uniformly mean ergodic but not power bounded (Example~\ref{contraexDual}). Finally, in Section 4, we study the dynamics of $C_{\varphi}$ on $\PmX_{\Vert \cdot \Vert}$. 
In contrast with what happens with the compact-open topology, in this case the properties of 
power boundedness and mean ergodicity are not related. We give examples of composition operators that are power bounded and not mean ergodic (Example~\ref{CS pb NOme}) and of operators that are mean ergodic and not power bounded (Example~\ref{CF me NOpb}). We also study the relation with Ces\`aro boundedness and prove that every mean ergodic operator on a Banach space is Ces\`aro bounded, but that there are  Cesàro bounded composition operators that are neither power bounded, nor mean ergodic (Example~\ref{CTacbNOpb}).
%
%
%
%
We use the theory of homogeneous polynomials and holomorphic functions as presented in \cite{Di99} and \cite{mujica}. For standard theory and notation of functional analysis we refer to \cite{MeVo97}.

\section{First results}

If we want to iterate the composition of a composiiton operator with itself we obviously need it to take values in $\mathcal{P}(^{m}X)$. This is the first thing that we have to settle, and we start with a simple observation.

\begin{nota}\label{PPIidDUal}
Suppose $X$ is a Banach space. If $x,y \in X$ satisfy that there are  $\gamma_{0} \in X'$ and $r>0$ such that $\gamma (x) = \gamma (y)$ for every $\gamma \in X'$ with $\Vert \gamma - \gamma_{0} \Vert <r$, then $x=y$.  Indeed, take  any $\phi \in X'$, fix $c > \Vert \phi \Vert$ and consider $\gamma: = \frac{r}{c} \phi + \gamma_0$. Then $\Vert \gamma - \gamma_{0} \Vert <r$ and
\[ 
 \frac{r}{c} \phi(x) + \gamma_0(x)= \gamma(x)=\gamma(y)= \frac{r}{c}\phi(y) + \gamma_0(y) \,.
\]
The fact that $\gamma_0(x)=\gamma_0(y)$ immediately gives $\phi (x) = \phi(y)$ and, since $\phi$ was arbitrary, $x=y$.
\end{nota}

\begin{prop}\label{Linealsiosi}
Let $\varphi:X\rightarrow X$ be a holomorphic mapping. The composition operator $C_\varphi:\PmX\rightarrow\PmX$ is well defined if and only if $\varphi$ is linear. 
\end{prop}
\begin{proof}
First, we assume that $\varphi:X\rightarrow X$ is linear (being holomorphic, it is continuous). If  $p\in \PmX$, we have
\[ 
C_\varphi(p)(\lambda x)= p(\varphi(\lambda x))
= p(\lambda\varphi(x)) =\lambda^{m} p(\varphi(x)) =\lambda^m C_\varphi(p)(x), 
\]
for all $x\in X$ and $\lambda\in\C$. Since $C_\varphi(p)$ is holomorphic, \cite[Corollary~15.34]{libro_2019} gives that it is an $m$-homogeneous polynomial and, therefore $C_\varphi: \PmX \to \PmX$ is well defined.\\
Suppose now that $C_\varphi : \mathcal{P}(^{m} X) \to \mathcal{P}(^{m} X)$ is well defined. This means that 
\[
p(\varphi(\lambda x))= p(\lambda\varphi(x)) =\lambda^m p(\varphi(x)),
\]
for all $p\in\PmX$, $\lambda\in\C$ and $x\in X$. Given $\gamma \in X'$ we have that $\gamma^{m}$, defined by $\gamma^{m}(x) = (\gamma (x))^{m}$, belongs to $\mathcal{P}(^{m} X)$. So, 
\begin{equation} \label{maggio}
\gamma(\varphi(\lambda x))^m=\lambda^m\gamma(\varphi(x))^m,
\end{equation}
for all $\lambda\in\C$ and $x\in X$. Then, for each $\gamma, \lambda, x$ there is some $\mu=\mu(\gamma,\lambda,x)\in\C$ with $\mu^m=1$ such that
\begin{equation} \label{botifarra}
\gamma(\varphi(\lambda x))= \mu \lambda\gamma(\varphi(x)).  
\end{equation}
Note that if  $\gamma(\varphi(x))=0$, then by \eqref{maggio} $\gamma(\lambda \varphi(x))=0$, and we can take  $\mu(\gamma,\lambda,x) =1$ for every $\lambda$ (in fact, in this case the equality holds for any value of $\mu$ we choose). Our aim si to show that we can also take $\mu(\gamma,\lambda,x) =1$ for every $\gamma,\lambda,x$. 
 To begin with we show that $\mu$ does not depend on $\gamma$ (i.e. $\mu =\mu(\lambda,x)$). Fix $x_0\in X$ and  $\gamma_0 \in X'$ so that $\gamma_0(\varphi(x_0))\neq 0$. Since $T : X' \to \mathbb{C}$ defined as
\[ 
T(\gamma):=\gamma(\varphi(x_{0}))
\]
is a well-defined continuous linear operator, given any $\varepsilon >0$ we find $r>0$ so that 
\[ 
|T(\gamma)| = \vert \gamma(\varphi(x_{0})) \vert >\varepsilon,
\]
for every $\gamma \in B(\gamma_{0}, r)$ (the open unit ball centred at $\gamma_{0}$ with radius $r$). We now fix $\lambda_{0} \in \mathbb{C} \setminus \{0\}$ and consider the function $f: B(\gamma_{0}, r) \to \mathbb{C}$ 
by
\[
f(\gamma)=\frac{\gamma(\varphi(\lambda_{0} x_0))}{\lambda_{0} \gamma(\varphi(x_0))} \,.
\]
This is continuous and $f(\gamma) = \mu = \mu(\gamma,\lambda_{0},x_{0})$ for every $\gamma \in B(\gamma_{0}, r)$. But $\mu$ is an $m$-th root of $1$, so $f$ takes values in a finite set an therefore has to be constant. In other words, there is some $\mu_{0} = \mu_{0} (\lambda_{0}, x_{0})$ so that $f(\gamma) = \mu_{0}$ for all  $\gamma \in B(\gamma_{0}, r)$, that is
\[
\gamma(\varphi(\lambda_{0} x_0))=\gamma(\lambda_{0}\mu_{0}\varphi(x_0)),
\]
for every $\gamma$ with $\Vert \gamma - \gamma_{0} \Vert < r$. Remark~\ref{PPIidDUal} yields 
\[
\varphi(\lambda_{0} x_{0})=  \mu_{0}\lambda_{0} \varphi(x_{0}) \,.
\]
This shows that for each $\lambda$ and $x$ there is  some $\mu = \mu (\lambda, x)$ such that \eqref{botifarra} holds for every $\gamma$.

Our next step is to see that $\mu$ can also be taken independently from $\lambda$. To do so, first we observe that the mapping $\lambda \in \mathbb{C} \rightsquigarrow \lambda \gamma_{0}(\varphi (x_{0})) \in \mathbb{C}$ is 
continuous (recall that $\gamma_{0}$ and $x_{0}$ are chosen so that $\gamma_0(\varphi(x_0))\neq 0$). Then the function $g : \mathbb{C} \setminus \{0\} \to \mathbb{C}$ given by
\[
g(\lambda) = \frac{\gamma_{0}(\varphi(\lambda x_0))}{\lambda \gamma_{0}(\varphi(x_0))} 
\]
is continuous and $g(\lambda) = \mu =\mu(\lambda , x_{0})$. As before, $g$ takes values on a finite set, hence is constant and we can find $\mu_{0}(x_{0})$ so that $\mu_{0} (\lambda , x_{0}) = \mu_{0}(x_{0})$ for every $\lambda \in \mathbb{C}$ (note that taking $\lambda =0$ in \eqref{maggio}, the equality in \eqref{botifarra} holds for any $\mu$). Then, for each fixed $x$ there is $\mu=\mu(x)$ so that \eqref{botifarra} holds for every $\gamma, \lambda$. In other words, given $\lambda$ and $x$ we have that $\gamma(\varphi(\lambda x))= \gamma(\mu(x) \lambda \varphi(x))$ for every $\gamma \in X'$ and, then
\[
\varphi(\lambda x) = \mu(x) \lambda \varphi(x), \,
\]
for every $\lambda \in \mathbb{C}$. Taking $\lambda =1$ shows that in \eqref{botifarra} we may take $\mu(x) = 1$ for every $x$. This shows our claim and 
\[
\gamma(\varphi(\lambda x))=  \lambda\gamma(\varphi(x)),
\]
for every $\gamma, \lambda, x$. Therefore, $\lambda \varphi(x)= \varphi(\lambda x)$ for every $\lambda, x$, as we have that $\gamma(\varphi(\lambda x))=\gamma(\lambda\varphi(x))$ for all $\gamma$.

Since $\varphi$ is holomorphic we can find a unique sequence $(p_{m})_{m}$, where each $p_{m}:X \to X$ is an $m$-homogeneous polynomial, satisfying \eqref{taylor}. Then
\[ 
\lambda\sum_{m=0}^{\infty} p_m(x)= \lambda \varphi(x)= \varphi(\lambda x)= \sum_{m=0}^{\infty} p_m(\lambda x)= \sum_{m=0}^{\infty} \lambda^m p_m(x), 
\]  
for every $\lambda \in \mathbb{C}$ and $x \in X$. The uniqueness of the sequence of polynomials yields $(\lambda^m-\lambda)p_m(x)=0$ for every $m \in \mathbb{N}_{0}$, $\lambda \in \mathbb{C}$ and $x \in X$. Taking any $\lambda^{m-1} \neq 1$ shows that $p_{m} \equiv 0$ for every $m \neq 1$ and therefore $\varphi = p_{1}$ is linear.
\end{proof}

We  study  dynamical properties of a composition operator $C_{\varphi} :  \mathcal{P}(^{m}X) \to  \mathcal{P}(^{m}X)$. By Proposition~\ref{Linealsiosi}, $\varphi$ is a continuous linear mapping. We also have

\begin{prop}
	Let $\varphi : X \to X$ be a continuous linear operator. If $\tau = \tau_{0}$ or $\Vert \cdot \Vert$, the composition operator $C_{\varphi} :  \mathcal{P}(^{m}X)_{\tau} \to  \mathcal{P}(^{m}X)_{\tau}$ is continuous.
\end{prop}
\begin{proof}
	If $\tau=\tau_{0}$, given any arbitrary compact subset $K\subset X$, the set $L:=\varphi(K)$ is also compact and
	\[ 
	\sup_{x\in K} |C_{\varphi}(p)(x)|= \sup_{x\in K} |p(\varphi(x))|= \sup_{x\in L} |p(x)|, 
	\]
	for all $p\in \PmX$.	If $\tau=\Vert \cdot \Vert$, we observe 
	\[
	\Vert C_{\varphi}(p) \Vert_{\mathcal{P}(^{m}X)} 
	= \sup_{\Vert x \Vert_{X} \le 1} |p(\varphi(x))|\\ 
	\le \|p\|_{\mathcal{P}(^{m}X)}\sup_{\Vert x \Vert_{X} \le  1} \| \varphi(x)\|_{X}^{m}
	\le \|\varphi\|_{L(X)}^{m}\ \|p\|_{\mathcal{P}(^{m}X)},
	\]
	for all $p\in \PmX$.
\end{proof}

\section{Dynamics with the compact-open topology}



Our first result is mentioned in~\cite[p.~917]{BoPaRi11}. We include a proof for completeness.  It relies on the fact \cite[Proposition~3.3]{BoPaRi11} that every power bounded operator on a semi-reflexive lcHs is mean ergodic (this extended an analogous result for reflexive Fr\'echet spaces \cite[Corollary~2.7]{AlBoRi09}).

\begin{prop} \label{fok}
	Let $E$ be a semi-Montel locally convex Hausdorff space. Then every power bounded operator on $E$ is uniformly mean ergodic.
\end{prop}
\begin{proof}
	Let $T\in L(E)$ be power bounded. Since $E$ is semi-Montel, it is semi-reflexive and, by \cite[Proposition 3.3]{BoPaRi11}, $T$ is mean ergodic. This means that the sequence $(T_{[n]})_{n}$  converges pointwise. Since $T$ is power bounded, $(T_{[n]})_{n}$ is equicontinuous. Hence, $S(x):=\lim_{n}T_{[n]}x$, for $x\in E$, defines an operator $S\in L(E)$.  Now, by \cite[(2), p.~139]{K2}, the topology of pointwise convergence and of uniform convergence on precompact sets coincide on  $(T_{[n]})_{n}$, which concludes the proof since every bounded set in $E$ is also precompact.
\end{proof}

Now, by \cite[Theorem~2.5]{MitsuhiroMiyagi}, the space $\PmX_{\tau_0}$ is (DFC). Hence, it is  semi-Montel~\cite[Definition~8.3.49]{perezcarrerasbonet} and so, we have

\begin{coro}\label{PmX powerbdd ume}
	Let $\varphi:X\rightarrow X$ be a continuous linear mapping. If $C_\varphi:\PmX_{\tau_0} \rightarrow \PmX_{\tau_0}$ is power bounded, then it is uniformly mean ergodic.
\end{coro}

The converse implication does not hold in general. To show this fact, we  characterise the power boundedness of the 
composition operator in terms of properties of the symbol $\varphi$. We begin with the following

\begin{lema}\label{polynomialconvex}
Let $K \subseteq X$ be a compact set and $m \geq 1$. Then the set
$$
\widehat{K}_{\PmX}:= \{ x\in X : |p(x)|\leq \sup_{y\in K} |p(y)|,\ \mbox{\rm for all } p\in \PmX  \} 
$$
is compact.
\end{lema}
\begin{proof}
First, we observe that $\widehat{K}_{\PmX}$ is closed, being an intersection of closed sets. Now, for every $\gamma \in X'$ and $x\in \widehat{K}_{\PmX}$, we have $\vert \gamma^{m} (x) \vert \leq \sup_{y\in K}|\gamma^{m} (y)|$ (because $\gamma^{m}\in \PmX$) and, consequently,
\[
 \bigg(\sup_{x\in \widehat{K}_{\PmX}}|\gamma(x)|\bigg)^m 
= \sup_{x\in \widehat{K}_{\PmX}}|\gamma^m(x)|  
\leq  \sup_{x\in K} |\gamma^m(x)| 
= \bigg(\sup_{x\in K}|\gamma(x)| \bigg)^m  \,.
\]
This implies $K^\circ  
\subseteq 
(\widehat{K}_{\PmX})^\circ$.
An application of Krein's theorem~\cite[(4), pg.~325]{K1} gives that the closure of the absolutely convex hull $\overline{\Gamma(K)}$ of $K$ is compact. Since, by the Bipolar theorem \cite[22.13]{MeVo97},
\[
\widehat{K}_{\PmX}\subseteq (\widehat{K}_{\PmX})^{\circ \circ}\subseteq K^{\circ \circ }\subseteq \big(\overline{\Gamma(K)}\big)^{\circ \circ }=\overline{\Gamma(K)},
\]
we obtain the result.
\end{proof}

Now, we characterise when a composition operator is power bounded. 
We say that a continuous linear mapping $\varphi : X \to X$ has \emph{stable orbits} (see \cite{BoDo2011A}) if for every compact set $K \subseteq X$ there is some compact set $L\subseteq X$ so that $\varphi^n(K) \subseteq L$ for every $n\in\N$.

\begin{prop}\label{PBDestablepolinomis}
Let $\varphi:X\rightarrow X$ be a continuous linear map. Then $C_\varphi:\PmX_{\tau_0} \rightarrow \PmX_{\tau_0}$ is power bounded if and only if $\varphi$ has stable orbits.
\end{prop}
\begin{proof}
Let us suppose first that $\varphi$ has stable orbits and fix $K \subseteq X$ compact. Then we can find a compact set $L \subseteq X$ so that $\varphi^n (x) \in L$ for every $x \in K$ and $n\in\N$. This  gives
$\sup_{x\in K} |C_\varphi^n(p)(x)|
\leq \sup_{x\in L} |p(x)|$
for every $p\in \PmX$ and $n \in \mathbb{N}$. Hence  $C_{\varphi}$ is power bounded.

Assume now that $C_{\varphi}$ is power bounded. If $\varphi$ does not have stable orbits there is a compact set $K\subset X$ such that $\cup_{n=0}^{\infty} \varphi^{n}(K)$ is not relatively compact. Since $(C_{\varphi}^{n})_{n}=(C_{\varphi^{n}})_{n}$ is equicontinuous in $L(\PmX)$, for the compact set $K$, we can find another compact set $W \subseteq X$ and  $c>0$ so that, for all $p\in\PmX$ and  $n\in \N,$
\begin{equation}\label{no-so}
\sup_{x\in K} |p(\varphi^n(x))| \leq c \sup_{x\in W}|p(x)| 
= \sup_{x\in W}|p(c^{1/m} x)| = \sup_{x\in c^{1/m} W}|p(x)| \,.
\end{equation}
The set $V:=c^{1/m} W$ is compact and, by Lemma~\ref{polynomialconvex}, and so also is $L:=\hat{V}_{\mathcal{P}(^{m}X)}$. If  there are $n_0\in\N$ and $x_0\in K$ so that $\varphi^{n_0}(x_0)\notin L$, then we can find $p\in\PmX$ such that $|p(\varphi^{n_0}(x_0))| > \sup_{y\in V}|p(y)|$. But this is not possible by \eqref{no-so}, which shows that $\varphi^{n}(K)\subseteq L$ for all $n\in\N$, contradicting the fact that $\cup_{n=0}^{\infty} \varphi^{n}(K)$ is not relatively  compact. This completes the proof.
\end{proof}


We  give an example of a composition operator showing that the converse implication in Corollary~\ref{PmX powerbdd ume} does not hold in general. This example and 
others that will be given later for $\mathcal{P}(^{m}X)_{\Vert \cdot \Vert}$ are based on the weighted backward shift, 
defined as follows. Fix $1 \leq p < \infty$ and
take $0 < \alpha < 1/p$. The \textit{unilateral weighted backward shift} is the operator $\varphi_{\alpha} : \ell_{p} \to \ell_{p}$ defined by
\begin{equation} \label{weightedbackward}
\varphi_{\alpha} (e_{1} ) =0 \, \text{ and }\, \varphi_{\alpha} (e_k) = \Big(\frac{k}{k-1}\Big)^\alpha e_{k-1} \, \text{ for } \, k \geq 2 \,,
\end{equation}
that is,  $\varphi_{\alpha} (x_{1}, x_{2}, \ldots) = (w_{1}x_{2}, w_{2}x_{3}, \ldots)$, for every $x=(x_{i})_{i}\in\ell_{p},$ where $w_{k}= \big(\frac{k+1}{k}\big)^\alpha,$ for $k\in\N$.

\begin{examp}\label{contraexDual}
For each $1 < p< \infty$ and $0<\alpha<1/p$ the composition operator $C_{\varphi_{\alpha}} : \mathcal{P}(^{1} \ell_{p})_{\tau_{0}} \to  \mathcal{P}(^{1} \ell_{p})_{\tau_{0}}$ is uniformly mean ergodic, but not power bounded. 

We recall that $\varphi_{\alpha}$ is mixing \cite[Corollary~2.3]{BonillaPeris} and, since $\ell_{p}$ is separable,  hypercyclic (see e.g. \cite[Theorem~1.2]{BaMa09}). This means that there exists $x_{0} \in \ell_{p}$ such that $\{ \varphi_{\alpha}^{n} (x_{0}) \}_{n}$ is dense in $\ell_{p}$. 
Since norm and weakly bounded sets coincide (see e.g. \cite[Theorem~2.5.5]{Me98} or \cite[Proposition~8.11]{MeVo97}) this implies that  $\{ \varphi_{\alpha}^{n} (x_{0}) \}_{n}$ is not weakly bounded, and we can find 
$u \in \ell'_{p} = \mathcal{P}(^{1} \ell_{p})$ such that $\{ |u(\varphi_{\alpha}^{n} (x_{0}) )| \}_{n \in\N}$ is unbounded. Since $u(\varphi_{\alpha}^{n} (x_{0})) = C_{\varphi_{\alpha}}^{n} (u)(x_{0})$,  the operator $C_{\varphi_{\alpha}}$ is not
power bounded. 

To see that  $C_{\varphi_{\alpha}}$ is mean ergodic, first we observe that, by \cite[Theorem~2.2]{BonillaPeris}, there is $c>0$ so that
\[ 
\Big\| \frac{1}{n} \sum_{k=0}^{n-1} \varphi_{\alpha}^{k}(x) \Big\|_{\ell_{p}} \leq c\|x\|_{\ell_{p}}, 
\]
for all $n\in\N$ and $x\in \ell_p$. Therefore, for $u \in \ell'_{p}$ and $x \in \ell_{p}$ we have
\[
\Big\vert \frac{1}{n} \sum_{k=0}^{n-1} C_{\varphi_{\alpha}}^{k} (u)(x) \Big\vert
= \Big\vert \frac{1}{n} \sum_{k=0}^{n-1} u( \varphi_{\alpha}^{k} x)  \Big\vert
\leq \Vert u \Vert_{\ell'_{p}}  \Big\| \frac{1}{n} \sum_{k=0}^{n-1} \varphi_{\alpha}^{k}(x) \Big\|_p 
\leq c\,\Vert u \Vert_{\ell'_{p}}   \|x\|_{\ell_{p}} \,.
\]
This shows that $\big( (C_{\varphi_{\alpha}})_{[n]}(u) \big)_{n}$ is equicontinuous for every fixed $u \in \ell'_{p}$. 

Since $\ell_{p}$ is reflexive, \cite[Theorem~2.2 and Corollary~2.7]{BonillaPeris} give that  $\varphi_{\alpha}$ is mean ergodic. Then we can find $\varphi\in L(\ell_{p})$ such that $\lim_{n} (\varphi_{\alpha})_{[n]} (x)=\varphi(x) \in \ell_{p}$ for every $x \in \ell_{p}$. If $u \in \ell'_{p}$ is fixed, continuity gives
\[
\lim_{N\to \infty} (C_{\varphi_{\alpha}})_{[N]}(u)(x) = \lim_{N\to \infty} u( {\varphi_{\alpha}}_{[N]}(x)) = u\big(\lim_{N\to \infty} {\varphi_{\alpha}}_{[N]} (x) \big)= u(\varphi (x)),
\]
for every $x\in\ell_{p}$. In other words, $\big( (C_{\varphi_{\alpha}})_{[n]} (u)\big)_{n}$ converges pointwise to $C_\varphi(u)$ for every $u\in\ell_{p}'$. Now, by \cite[(2), p.~139]{K2} the topology of pointwise convergence and of convergence on compact sets coincide on equicontinuous sets. Since $C_\varphi\in L\big((\ell'_{p})_{\tau_0}\big)$, $\big( (C_{\varphi_{\alpha}})_{[n]} (u) \big)_{n}$ is $\tau_{0}$-convergent to $C_\varphi(u)$ for every $u\in\ell_{p}'$ and, hence, $C_{\varphi_{\alpha}}$ is mean ergodic.

In fact,  $C_{\varphi_{\alpha}}$ is uniformly mean ergodic. To check this  first we observe that $\big( (\varphi_{\alpha})_{[n]} - \varphi \big)_{n}$ is pointwise convergent to $0$ and so, equicontinuous on $\ell_p$. Therefore $\big( (\varphi_{\alpha})_{[n]} - \varphi \big)_{n}$ converges to $0$ uniformly on the compact subsets of $\ell_p$. Now,  we take an arbitrary $\tau_0$-bounded set $V\subset \ell'_{p}$, which is also norm-bounded in $\ell'_{p}$ (see, for instance, \cite[p.~267]{MeVo97}). Therefore, for any compact set $K\subset \ell_p$ and $n\in\N$ we have, for some constant $c>0$,
\begin{multline*}
 \sup_{u\in V}\sup_{x\in K}  \big| \big( (C_{\varphi_{\alpha}})_{[n]}-C_{\varphi} \big)(u)(x) \big|=\sup_{u\in V}\sup_{x\in K} \big| u\big( ((\varphi_{\alpha})_{[n]} - \varphi) (x)\big) \big|\\
\leq \sup_{u\in V}\sup_{x\in K} \|u\|_{p'} \big\| \big( (\varphi_{\alpha})_{[n]} - \varphi\big) (x) \big\|_p
 \leq c\sup_{x\in K}\big\| \big( (\varphi_{\alpha})_{[n]} - \varphi\big) (x) \big\|_p,
\end{multline*}
which gives the conclusion.
\end{examp}

\section{Dynamics with the norm topology}


We consider now the Banach space $\mathcal{P} (^{m} X)$ endowed with the norm given in \eqref{norma-p}. We study the interplay between power boundedness, Ces\`aro boundedness and mean ergodicity.  
%
%
%
%
%
%

As a first step we characterise, as we did in Proposition~\ref{PBDestablepolinomis}, the power boundedness of a composition operator by means of the symbol.

\begin{prop}\label{PBennormaPmX}
Let $\varphi:X\rightarrow X$ be a continuous linear map. Then $C_\varphi:\PmX_{\|\cdot\|}\to \PmX_{\|\cdot\|}$ is power bounded if and only if $\varphi$ is power bounded.
\end{prop}
\begin{proof}
Suppose in first place that $\varphi:X\rightarrow X$ is power bounded, then there is a constant $c>0$ such that 
\[ 
\|\varphi^n(x)\|_{X} \leq c \|x\|_{X}, 
\]
for all $n\in\N$ and for all $x\in X$. Using this we have, for $p\in\PmX$ and $n\in\N$,
\[
\|C_{\varphi^n}(p)\|_{\mathcal{P}(^{m}X)} 
\leq \|p\|_{\mathcal{P}(^{m}X)}  \sup_{\|x\|\leq 1} \| \varphi^n(x) \|_{X}^m \\
\leq 
c^m \|p\|_{\mathcal{P}(^{m}X)}.
\]	
Hence $C_\varphi:\PmX_{\|\cdot\|}\rightarrow \PmX_{\|\cdot\|}$ is power bounded. 

Conversely,  assume that $C_\varphi:\PmX_{\|\cdot\|}\rightarrow \PmX_{\|\cdot\|}$ is power bounded. We can find $c>0$ such that $\Vert p \circ \varphi^{n} \Vert \leq c \Vert p \Vert$, for every $p\in\PmX$. In particular, we have
\[
\sup_{\Vert x \Vert_{X} < 1} \vert \gamma (\varphi^{n} x) \vert^{m}
\leq c \sup_{\Vert z \Vert_{X} < 1} \vert \gamma (z) \vert^{m}, \,
\]
for every $n \in \mathbb{N}$ and every $\gamma \in X'$. Hence, we obtain  $\vert \gamma (\varphi^{n}x) \vert \leq c^{1/m} \Vert \gamma \Vert$ for every $\gamma \in X'$, and every  $x$ with $\Vert x \Vert_{X} <1$ and all $n\in\N$. An application of the Hahn-Banach theorem completes the proof.
\end{proof}

\begin{prop}
	Let $\varphi:X\rightarrow X$ be a continuous linear map such that $C_\varphi:\PmX_{\tau_0}\rightarrow \PmX_{\tau_0}$ is power bounded. Then $C_\varphi:\PmX_{\|\cdot\|}\rightarrow \PmX_{\|\cdot\|}$ is power bounded.
\end{prop}
\begin{proof}
By Proposition~\ref{PBDestablepolinomis}, $\varphi$ has stable orbits and for each $x \in X$ we can find a compact set $K_{x} \subseteq X$ such that $(\varphi^n(x))_{n\in \N} \subset K_x$. This gives that 
$\sup_{n\in\N} \| \varphi^n(x)\| < \infty$ for every $x \in X$ and, by the uniform boundedness principle, $\sup_{n\in\N} \| \varphi^n \| < \infty$. This shows that $\varphi$ is power bounded and, by Proposition~\ref{PBennormaPmX}, so also is 
$C_\varphi:\PmX_{\|\cdot\|}\rightarrow \PmX_{\|\cdot\|}$.
\end{proof}

The converse implication is not true in general.
\begin{examp}
Consider the composition operator $C_\varphi : \mathcal{P}(^{m}c_{0}) \to \mathcal{P}(^{m}c_{0})$ defined by the usual forward shift $\varphi: c_0 \rightarrow c_0$ given by $\varphi(x)=(0, x_1, x_2, \dots)$.
Let us see that  $C_\varphi$  is power bounded in $\PmX_{\|\cdot\|}$ but it is not in $\PmX_{\tau_0}$. On the one hand,  we observe that $\|\varphi^n(x)\| =\|x\|$ for every $x \in c_{0}$ and all $n\in\N$, but 
$(\varphi^n(e_1))_{n\in\N}=(e_n)_{n\in\N}$ is not relatively compact in $c_0$. This shows that $\varphi$ is power bounded but does not have stable orbits. As a consequence of Propositions~\ref{PBennormaPmX} and~\ref{PBDestablepolinomis}, $C_{\varphi}$ is power bounded on $\mathcal{P}(^{m} c_{0})_{\Vert \cdot \Vert}$ but not on $\mathcal{P}(^{m} c_{0})_{\tau_{0}}$.
\end{examp}

\begin{examp}\label{CTacbNOpb}
Fix $m \geq 2$ and  $0<\alpha<1/m$. The operator $\varphi_{\alpha} : \ell_{m} \to \ell_{m}$ defined in \eqref{weightedbackward} satisfies that $C_{\varphi_{\alpha}} : \mathcal{P}(^{m} \ell_{m})_{\Vert \cdot \Vert} \to  \mathcal{P}(^{m} \ell_{m})_{\Vert \cdot \Vert}$ is Ces\`aro bounded but neither power bounded nor mean ergodic.

From the proof of \cite[Theorem~2.1]{BonillaPeris} we obtain
\[
\sum_{k=1}^{n} \Vert \varphi_{\alpha}^{k} (x) \Vert_{\ell_{m}}^{m} \leq 4n,
\]
for every $x \in \ell_{m}$, $\|x\|\le 1$, and $n \in \mathbb{N}$. Hence given $p \in \mathcal{P}(^{m} \ell_{m})$ we have, for $\|x\|\le 1$,
\[
\Big\Vert \sum_{k=1}^{n} p (\varphi_{\alpha}^{k} x) \Big\Vert_{\ell_{m}} 
\leq \sum_{k=1}^{n}  \Vert p (\varphi_{\alpha}^{k} x) \Vert_{\ell_{m}} 
\leq \sum_{k=1}^{n}  \Vert p \Vert_{\mathcal{P}(^{m}\ell_{m})} \Vert \varphi_{\alpha}^{k} (x) \Vert_{\ell_{m}} 
\leq \Vert p \Vert_{\mathcal{P}(^{m}\ell_{m})} 4 n \,.
\]
Now, we take the supremum over $\|x\|\le 1$ to obtain
\[
\Big\Vert \frac{1}{n} \sum_{k=1}^{n} C_{\varphi_{\alpha}}^{k} (p) \Big\Vert_{\mathcal{P}(^{m}\ell_{m})}
\leq 4  \Vert p \Vert_{\mathcal{P}(^{m}\ell_{m})},
\]
which shows that $ C_{\varphi_{\alpha}}$ is Ces\`aro bounded.
We know that $\varphi_{\alpha}$ is hypercyclic~\cite{BonillaPeris} (see also Example~\ref{contraexDual} in the present notes). Hence it cannot be power bounded and so, by Proposition~\ref{PBennormaPmX}, neither is $C_{\varphi_{\alpha}}$. To show that it is not mean ergodic we take the $m$-homogeneous polynomial given by
\begin{equation} \label{sparss}
p(x)=\sum_{i\in\N} x_i^m,
\end{equation}
and prove that $\big( (C_{\varphi_{\alpha}})_{[n]}(p) \big)_{n\in\N}$ does not converge in $\mathcal{P}(^{m} X)$. First, we observe that, for a fixed $n\in\N$, we have
\[
\varphi_{\alpha}^{k}( e_{n+1})= \begin{cases}
\big(\frac{n+1}{n+1-k}\big)^\alpha e_{n+1-k}, & \text{ if } n \geq k, \\
0, & \text{ if } n< k.
\end{cases}
\]
Then
\begin{align*}
|(C_{\varphi_{\alpha}})_{[n]}(p)(e_{n+1}) &-(C_{\varphi_{\alpha}})_{[3n]}(p)(e_{n+1})| \\
&= \left|  \frac{1}{n} \sum_{k=1}^n p(\varphi_{\alpha}^k (e_{n+1}) ) -\frac{1}{3n} \sum_{k=1}^{3n} p(\varphi_{\alpha}^k (e_{n+1}) )  \right|\\
&= \left|  \frac{1}{n} \sum_{k=1}^n p(\varphi_{\alpha}^k (e_{n+1}) ) -\frac{1}{3n} \sum_{k=1}^{n} p(\varphi_{\alpha}^k (e_{n+1}) )  \right|\\
&= \frac{2}{3n} \left|  \sum_{k=1}^n p(\varphi_{\alpha}^k (e_{n+1}) ) \right| = 
\frac{2}{3n} \left|  \sum_{k=1}^n p\left(\left(\frac{n+1}{n+1-k}\right)^\alpha e_{n+1-k} \right) \right| \\
&= \frac{2}{3n} \left|  \sum_{k=1}^n \left(\frac{n+1}{n+1-k}\right)^{m \alpha} \right|
= \frac{2}{3n} \left|  \sum_{k=1}^n \left(1+ \frac{k}{n+1-k}\right)^{m \alpha} \right| \\
&\geq \frac{2}{3n} \left|  \sum_{k=1}^n 1 \right| = \frac{2}{3}.
\end{align*}
This implies
\[
\Vert (C_{\varphi_{\alpha}})_{[n]}(p) -(C_{\varphi_{\alpha}})_{[3n]}(p) \Vert > \frac{2}{3}\,,
\]
and so, $\big( (C_{\varphi_{\alpha}})_{[n]} \big)_{n}$ is not Cauchy. Hence $C_{\varphi_{\alpha}}$ is not mean ergodic.
\end{examp}

This settles the relationship between absolute Ces\`aro boundedness and power boundedness and mean ergodicity. We look now at the latter two. Unlike what we saw in Corollary~\ref{PmX powerbdd ume} for the compact-open topology, when we consider the norm topology we may find composition operators that are power bounded but not mean ergodic.

\begin{examp}\label{CS pb NOme}
For $m \geq 1$ we consider the usual backward shift $\sigma : \ell_{m} \to  \ell_{m}$ defined as
\[
\sigma (x_{1}, x_{2}, x_{3}, \ldots) = (x_{2}, x_{3}, \ldots) \,.
\]
Then the composition operator $C_{\sigma} : \mathcal{P} (^{m} \ell_{m})_{\Vert \cdot \Vert} \to \mathcal{P} (^{m} \ell_{m})_{\Vert \cdot \Vert}$ is power bounded but not mean ergodic.

Let us observe that $\Vert \sigma^{n} (x) \Vert_{\ell_{m}} \leq \Vert x \Vert_{\ell_{m}}$ for every $x\in\ell_{m}$. So $\sigma$ is power bounded. Applying Proposition~\ref{PBennormaPmX} we obtain that $C_{\sigma}$ is power bounded.

To see that it is not mean ergodic we take the polynomial $p$ defined in \eqref{sparss} and observe that
\[
\sigma^{k} (e_{n+1}) = \begin{cases}
e_{n+1-k}, & \text{ if } n \geq k, \\
0, & \text{ if } n < k.
\end{cases}
\]
Then
\begin{multline*}
|(C_{\sigma})_{[n]}(p)(e_{n+1})-(C_{\sigma})_{[3n]}(p)(e_{n+1})| \\
= \left|  \frac{1}{n} \sum_{k=1}^n p(\sigma^k( e_{n+1}) ) -\frac{1}{3n} \sum_{k=1}^{3n} p(\sigma^k (e_{n+1}) )  \right|
= \left|  \frac{1}{n} \sum_{k=1}^n p(\sigma^k (e_{n+1} )) -\frac{1}{3n} \sum_{k=1}^{n} p(\sigma^k (e_{n+1}) )  \right|\\
= \frac{2}{3n} \left|  \sum_{k=1}^n p(\sigma^k (e_{n+1}) ) \right| = 
\frac{2}{3n} \left|  \sum_{k=1}^n p\left( e_{n+1-k} \right) \right| 
= \frac{2}{3n} \left|  \sum_{k=1}^n 1 \right|= \frac{2}{3}.
\end{multline*}
This, as in Example~\ref{CTacbNOpb}, shows that  $\big( (C_{\sigma})_{[n]} \big)_{n}$ is not Cauchy and that $C_{\sigma}$ is not mean ergodic.
\end{examp}

\begin{examp}\label{CF me NOpb}
For fixed $1 < p < \infty$ we take $0 < \alpha < \frac{1}{p'} : = 1 - \frac{1}{p}$ and define the weighted forward shift $\psi_{\alpha} : \ell_{p} \to \ell_{p}$ by
\[
\psi_{\alpha} (e_{k}) = \Big(  \frac{k}{k-1} \Big)^{\alpha} e_{k+1} \text{ for } k \geq 1 \,,
\]
that is, $\psi_{\alpha} (x_{1}, x_{2}, \ldots) = (0, w_{1} x_{1}, w_{2}x_{2}, \ldots)$, where, as before,  $w_{k} =\big(  \frac{k+1}{k} \big)^{\alpha}$ for $k\ge 1$. We consider the composition 
operator $C_{\psi_{\alpha}} : \mathcal{P}(^{1} \ell_{p})_{\Vert \cdot \Vert} \to \mathcal{P}(^{1}\ell_{p})_{\Vert \cdot \Vert}$. Since $\mathcal{P}(^{1}\ell_{p})_{\Vert \cdot \Vert} = \ell'_{p}= \ell_{p'}$, for each
$u \in \ell_{p'}$ and $x \in \ell_{p}$ we have
\[
C_{\psi_{\alpha}} (u)(x) = u\big( {\psi_{\alpha}}(x) \big) = u(0, w_{1} x_{1}, w_{2}x_{2}, \ldots)
= \sum_{k=1}^{\infty} x_{k} w_{k} u_{k+1} = \varphi_{\alpha}(u)(x) \,.
\]
This shows that $C_{\psi_{\alpha}} = \varphi_{\alpha} : \ell_{p'} \to \ell_{p'}$. By \cite[Corollary~2.7]{BonillaPeris} ($\ell_{p}$ is reflexive), $C_{\psi_{\alpha}}$ is mean ergodic. On the 
other hand, it is hypercyclic \cite[Corollary 2.3]{BonillaPeris} (see also Example~\ref{contraexDual} in the present notes) and therefore not power bounded.
\end{examp}


%

\textbf{Acknowledgements.} We are indebted to  Prof. Jos\'e Bonet and Prof. Daniel Carando for helpful suggestions about this work. The research of the first author was partially supported by the project MTM2016-76647-P. The research of the second author was partially supported by the project GV Pro\-meteo 2017/102. The research of the third author was partially supported by the project  MTM2017-83262-C2-1-P.

\end{document}